\documentclass[11pt]{amsart}
\usepackage{extpfeil}
\usepackage[all]{xy}
\usepackage{amssymb} 
\usepackage[margin=1in]{geometry}
\usepackage{euscript}
\usepackage{mathrsfs}
\usepackage{mathtools}
\usepackage{amsthm}
\usepackage{mathrsfs}
\usepackage{url}
\usepackage{color}

\numberwithin{equation}{section}

\theoremstyle{plain}
	\newtheorem{thm}[equation]{Theorem}
	\newtheorem{prop}[equation]{Proposition}
	\newtheorem{lem}[equation]{Lemma}

	\newtheorem{cor}[equation]{Corollary}
	\newtheorem{lem/defn}[equation]{Lemma/Definition}

\theoremstyle{definition}
	\newtheorem{defn}[equation]{Definition}

\theoremstyle{remark}
	\newtheorem{rem}[equation]{Remark}

\def\dg{{\mathrm dg}}
\def\Moddg{\operatorname{Mod}_\dg}

\def\nc{\newcommand}
\def\on{\operatorname}
\def\mf{\on{mf}}

\def\builds{\models}
\def\cT{{\mathcal T}}
\def\cS{{\mathcal S}}
\def\cX{{\mathcal X}}
\def\Thick{\on{Thick}}
\nc{\edit}[1]{\marginpar{\footnotesize{#1}}}

\nc{\C}{\mathbb{C}}
\nc{\Q}{\mathbb{Q}}
\nc{\Z}{\mathbb{Z}}
\nc{\PP}{\mathbb{P}}
\nc{\R}{\mathbb{R}}
\nc{\LL}{\mathbb{L}}
\nc{\OO}{\mathcal{O}}

\nc{\X}{\EuScript{X}}
\nc{\cC}{\EuScript{C}}
\nc{\cE}{\EuScript{E}}
\nc{\cA}{\EuScript{A}}
\nc{\sZ}{\EuScript{Z}}

\nc{\id}{{\on{id}}}
\nc\Hom{{\on{Hom}}}
\nc\cone{{\on{cone}}}
\nc\Ob{{\on{Ob}}}
\nc\Spec{{\on{Spec}}}
\nc\Mod{{\on{Mod}}}
\nc\Perf{{\on{Perf}}}
\nc\End{{\on{End}}}
\nc{\into}{\hookrightarrow}
\nc{\tr}{\on{tr}}
\nc{\ev}{\on{ev}}
\nc{\im}{\on{im}}
\nc{\Mot}{\on{Mot}}
\nc{\pt}{\on{pt}}
\nc{\coker}{\on{coker}}
\nc{\rk}{\on{rank}}
\nc{\TOP}{\on{Top}_{\mathbb{C}}^{s}}
\nc{\gr}{\on{gr}}
\nc{\Catperf}{\text{Cat}^{\text{perf}}}
\nc{\Sym}{\on{Sym}}
\nc{\xra}{\xrightarrow}
\nc{\lra}{\xleftarrow}
\nc{\Bet}{\mathbf{Betti}_{X}}
\nc{\codim}{\on{codim}}
\nc{\Fred}{\on{Fred}}
\nc{\colim}{\on{colim}}
\nc{\KK}{{\bf K}}
\nc{\Sp}{\on{Sp}}
\nc{\onto}{\twoheadrightarrow}
\nc{\A}{\mathbb{A}}
\nc{\Aff}{\on{Aff}}
\nc{\SH}{\on{SH}}
\nc{\QCoh}{\on{QCoh}}
\nc{\Alg}{\on{Alg}}
\nc{\Br}{\on{Br}}
\nc{\ta}{\widetilde{\a}}
\nc{\Shv}{\on{Shv}}
\nc{\GG}{\mathbb{G}}
\nc{\red}{\color{red}}
\nc{\an}{\on{an}}
\nc{\D}{\on{D}}
\nc{\qc}{\on{qc}}
\nc{\op}{{\on{op}}}
\nc{\shEnd}{{\mathcal End}}
\nc{\Sph}{\mathbb{S}}
\nc{\Top}{\on{Top}}
\nc{\Map}{\on{Map}}
\nc{\Vect}{\on{Vect}}
\nc{\holim}{\on{holim}}
\nc{\GL}{\on{GL}}

\newcommand\blfootnote[1]{%
  \begingroup
  \renewcommand\thefootnote{}\footnote{#1}%
  \addtocounter{footnote}{-1}%
  \endgroup
}

\def\a{\alpha}

\def\Perf{\on{Perf}}

\def\Sp{\on{Sp}}

\def\cC{\mathcal{C}}
\def\cD{\mathcal{D}}

\def\QCoh{\on{QCoh}}

\def\an{\on{an}}
\def\nc{\on{nc}}

\def\fm{\mathfrak{m}}

\def\wQ{\widehat{Q}}

\def\m{\mathfrak{m}}
\def\stab{{\on{stab}}}

\title{Idempotent completions of equivariant matrix factorization categories}
\author{Michael K. Brown}
\author{Mark E. Walker}
\date{}

\makeatletter
\@namedef{subjclassname@2020}{%
  \textup{2020} Mathematics Subject Classification}
\makeatother

\begin{document}
\setcounter{page}{1} \thispagestyle{empty}
\maketitle
\blfootnote{The second author was partially supported by National   Science Foundation grant DMS-2200732.}
\begin{abstract}
We prove that equivariant matrix factorization categories associated to henselian local hypersurface rings are idempotent complete, generalizing a result of Dyckerhoff in the non-equivariant case. 
\end{abstract}
\section{Introduction}

It follows from a result of Dyckerhoff \cite[Lemma 5.6]{dyckerhoff} that matrix factorization categories associated to complete local hypersurface rings are idempotent complete.
In this paper, we generalize this result to the equivariant case.

Throughout the paper, we let $G$ be a finite group acting on a
noetherian local ring $(Q,\m)$ in such a way that $Q$ is module finite
over the invariant subring $Q^G$,
and we assume $f$ is  an element of $\m$ that is fixed by $G$.
(The assumption that $Q^G \into Q$ is module finite holds quite generally; see Remark~\ref{finiteextension}.)  We write $[\mf_G(Q,f)]$ for the (triangulated) homotopy category of $G$-equivariant matrix factorizations of $f$;
see \S\ref{mf} for the definition. Our main result is:

\begin{thm} \label{main}
In the setting above: if $Q$ is henselian, then $[\mf_G(Q, f)]$ is idempotent complete.
\end{thm}

Suppose now that $Q$ is regular, $(Q', \fm')$ is another regular local ring with $G$-action, and $\phi: Q \to Q'$ is a $G$-equivariant  homomorphism of local rings.
Setting $f' = \phi(f)$, we have an induced triangulated functor
$$
\phi_*: [\mf_G(Q,f)] \to [\mf_G(Q',f')]
$$
given by extension of scalars along $\phi$. 
Building from the aforementioned result of Dyckerhoff \cite[Lemma 5.6]{dyckerhoff}, we also prove: 
\begin{prop} \label{introprop} Assume $Q/f$ and $Q'/f'$ have isolated singularities, $Q^G \subseteq Q$ and 
$(Q')^G \subseteq Q'$ are module finite, $|G|$ is a unit in $Q$,  $\phi$ is flat, and
the canonical map $Q' \otimes_Q Q/\fm \to Q'/\fm'$ is an isomorphism.  
The  functor $\phi_*$ induces an equivalence of triangulated categories
$$
\phi_*: [\mf_G(Q,f)]^\vee \xra{\cong} [\mf_G(Q',f')]^\vee,
$$
where ${}^\vee$ denotes idempotent completion.
\end{prop}

From Theorem~\ref{main} and Proposition~\ref{introprop}, we deduce:

\begin{cor} 
\label{introcor}
If $|G|$ is a unit in $Q$, $Q^G \subseteq Q$ is module finite, and the local hypersurface $Q/f$ has an isolated singularity, then  the canonical functors
$$
[\mf_G(Q, f)]^\vee \xra{} [\mf_G(Q^h, f)] \xra{} [\mf_G(\wQ, f)]
$$
are equivalences of triangulated categories, where $Q^h$ and $\wQ$ are the henselization and $\fm$-adic completion of $Q$, respectively. 
\end{cor}

As another application, we combine Theorem~\ref{main} and a result of Spellmann-Young \cite{SY} to conclude that $[\mf_G(Q, f)]$ is equivalent to the category of $G$-equivariant objects in the triangulated category $[\mf(Q, f)]$; see Corollary~\ref{fixedpoints} below for the precise (and more general)
statement. This gives an analogue of a result of Elagin involving bounded derived categories of equivariant sheaves \cite[Theorem 9.6]{elagin1}. This consequence of the statement
of Theorem~\ref{main} was observed by Spellman-Young  \cite[Remark 3.7]{SY} and provided a main source of motivation for this work.

\subsection*{Acknowledgments} 
We thank Jack Jeffries and Anurag Singh for suggesting the argument in the proof of Proposition~\ref{newprop}, and we are grateful to the anonymous referee for many helpful suggestions.

\section{Background}

\subsection{Twisted group rings} 
Let $A$ be a commutative ring with action of a finite group $G$.

\begin{defn} The {\em twisted group ring}, written  $A \# G$, has underlying set given by formal sums $\sum_{g \in G} a_g g$, with $a_g \in A$ for all $g$, and multiplication determined by the rule
$$
ag \cdot bh = ab^ggh
$$
for $a,b \in A$ and  $g,h \in G$, where $b^g$ is the result of acting by $g$ on $b$.
\end{defn}

The map $A \to A \# G$ sending $a$ to $ae_G$ is a ring homomorphism, but beware that $A \# G$ is only an $A$-algebra when the action of $G$ on $A$ is trivial, in which case $A \#
G$ coincides with the group ring $A[G]$.
In general, letting $A^G$ denote the ring of invariants $\{a \in A \mid a^g = a \text{ for all $g \in G$}\}$, the composition $A^G \into A \to A \# G$ exhibits $A \# G$ as an
$A^G$-algebra.

A left module over $A \# G$ is the same thing as a set $M$ that is equipped with a left $A$-module structure and a left $G$-action such that
$g (a m) = a^g (g m)$ for all $g \in G$, $a \in A$ and $m \in M$.

Suppose $A$ is local with maximal ideal $\m$. Since $G$ is finite, the
inclusion $A^G \into A$ is integral, 
and it is therefore a consequence of the Going Up Theorem that the invariant
ring $A^G$ is also local. 
We observe also that the henselization $A^h$ and $\m$-adic completion
$\widehat{A}$ inherit canonical $G$-actions. In more detail, every
morphism of local rings $\phi: A \to B$ induces unique morphisms on
henselizations $\phi^h: A^h \to B^h$ and completions $\widehat{\phi}:
\widehat{A} \to \widehat{B}$ 
that cause the evident squares to commute. Applying this when $B = A$ and $\phi$ ranges over the isomorphisms determined by the actions of the  group elements of $G$
gives the actions of $G$ on $A^h$ and $\widehat{A}$.

\begin{prop} \label{newprop}
Assume a finite group $G$ acts on a local ring $A$ in such a way that the extension $A^G \subseteq A$ is module finite.  Both of the extensions
  $(A^h)^G \subseteq A^h$ and $(\widehat{A})^G  \subseteq \widehat{A}$ are module finite, $(A^h)^G$ is henselian, and $(\widehat{A})^G$ is complete. 
\end{prop}

\begin{rem}
\label{finiteextension}
The assumption that the extension $A^G\subseteq A$ is module finite holds in many cases of interest.
For instance, if $A$ is a finite type $F$-algebra for a field $F$ contained in $A^G$, then, since $A^G \subseteq A$ is integral and finite type, it is module finite. 
More generally, if $A$ is any equivariant localization of an example of this kind,  then $A^G \subseteq A$ is module finite.  

Additionally, if $A$ is a noetherian domain, and $|G|$ is invertible in $A$, then $A^G \subseteq A$ is
module finite \cite[Proposition 5.4]{LW}. More generally, if $A$ is any equivariant quotient of an example of this kind, 
then $A^G \subseteq A$ is module finite. 

Indeed, we know of no examples where  $A^G \subseteq A$ fails to be module finite when $|G|$ is invertible in $A$, but there are examples of such failure
when $A = F[[x,y]]$, $F$ is a field of infinite transcendence degree over the field with $p$ elements for a prime  $p$, and $G$ is cyclic of order $p$; see \cite{GS}.
\end{rem}

\begin{proof}[Proof of Proposition~\ref{newprop}]
  For any module finite extension of local rings  $B \subseteq A$,
 the induced map $B^h \subseteq A^h$ is also a module finite extension, and the canonical map $B^h \otimes_B A \xra{\cong} A^h$ is an isomorphism \cite[Lemma 10.156.1]{stacksproject}.
  Applying this when $B = A^G$, and using that $(A^G)^h \subseteq (A^h)^G \subseteq A^h$, we obtain the first result. Similarly, we have $\widehat{A} \cong \widehat{A^G} \otimes_{A^G} A$ is module finite over $\widehat{A^G}$,
  and $\widehat{A^G} \subseteq (\widehat{A})^G \subseteq \widehat{A}$, hence $(\widehat{A})^G  \subseteq \widehat{A}$ is module finite.

 Let us  return to the general setting of a module finite extension of local rings
  $B \subseteq A$.  If $A$ is complete, then $B$ is complete. To see this, note that $\widehat{B} \otimes_B A \cong \widehat{A}  = A$, and hence $B \subseteq \widehat{B} \subseteq A$,
  so that $\widehat{B}$ is a module finite extension of $B$. Thus, $\widehat{B} \otimes_B (\widehat{B}/B)  = 0$. Since $\widehat{B}$ is a faithfully flat $B$-module, it follows that $B = \widehat{B}$. Likewise, if $A$ is henselian, so is $B$. Indeed, we have $B \subseteq B^h \subseteq A$, since $B^h \subseteq A^h$ and $A = A^h$. Thus, $B^h$ is module finite over $B$.
Taking completions, and using that $\widehat{(B^h)} \cong \widehat{B}$, we get $\widehat{B} \otimes_B B^h/B = 0$, and hence $B = B^h$. 
  \end{proof}

\section{Proof of the Main Theorem}
\label{mf}

Let us first recall the notion of an idempotent complete additive category:

\begin{defn}
\label{def:idem}
An additive category $A$ is \emph{idempotent complete} if, given an object $X$ in $A$ and an idempotent endomorphism $e$ of $X$, there exists an object $Y$ in $A$ and morphisms $\pi \colon X \to Y$, $\iota \colon Y \to X$ such that $\pi \iota = \id_Y$ and $\iota \pi = e$. The \emph{idempotent completion of $A$} is, roughly speaking, the smallest idempotent complete additive category containing $A$. More precisely: the idempotent completion $A^\vee$ of $A$ is the category with objects given by all pairs $(X, e)$, where $X$ is an object in $A$ and $e$ is an idempotent endomorphism of $X$. A morphism $(X, e) \to (X', e')$ in $A^\vee$ is a morphism $f \colon X \to X'$ in $A$ such that $e' f = fe = f$.
\end{defn}

Recall that a not-necessarily-commutative ring $E$ is called {\em nc local} if $E/J(E)$ is a division ring, where $J(E)$ denotes the Jacobson radical of $E$ (i.e., the intersection of all the
maximal left ideals of $R$). An additive category $A$ is called {\em Krull-Schmidt} 
if every object is a finite direct sum of objects with nc local endomorphism rings.
By a result of Krause \cite[Corollary 4.4]{krause}, every Krull-Schmidt additive category is  idempotent complete. We observe that if $A$ is a Krull-Schmidt additive category and $B$ is a quotient of $A$,
by which we mean $B$ has the same objects of $A$ and the hom groups of $B$ are quotients of the hom groups of $A$, then $B$ is also Krull-Schmidt.
In particular, any quotient of a Krull-Schmidt additive category is idempotent complete.

\begin{lem} \label{lem126} Suppose $A$ is an $R$-linear additive category, where $R$ is a henselian (local noetherian) ring.
  If $A$ is idempotent complete, and the endomorphism ring of every object of $A$ is finitely generated as an $R$-module,  then
 $A$ is Krull-Schmidt.
\end{lem}

\begin{proof} Our argument follows the proof of \cite[1.8]{LW}.
The assumptions imply that the endomorphism ring of every object is noetherian, 
and it follows that every object of $A$ is a finite direct sum of   indecomposable objects.
  Given any indecomposable object $X$ of $A$, set  $E \coloneqq \End_A(X)$.
  By assumption, $E$ is a module finite $R$-algebra, and so,  since $R$ is henselian, every idempotent of $E/J(E)$ lifts to an idempotent of $E$ \cite[A.30]{LW}.
  Since $X$ is indecomposable and $A$ is idempotent complete, $E$ has no nontrivial idempotents. We conclude that $E/J(E)$ has no nontrivial idempotents.
Again using that $E$ is module finite over $R$, by \cite[1.7]{LW} we have $\fm_R E\subset J(E)$ and thus $E/J(E)$ is a module finite algebra over the field
$R/\fm_R$. This shows $E/J(E)$ is artinian and hence semi-simple. Since it has no nontrivial idempotents, it must be a division ring.
\end{proof}

For a group $G$ acting on a commutative ring $Q$ and an element  $f \in Q$ fixed by the action, 
we write $\mf_G(Q,f)$ for  the additive category of equivariant matrix factorizations.
Objects are pairs $P = (P, d)$ with $P$ a $\Z/2$-graded module over the twisted group ring $Q \# G$ that is finitely generated and projective as a $Q$-module
and $d$ a $Q \# G$-linear endomorphism of $P$ of odd degree that squares to $f \cdot \id_P$. (We do not assume $|G|$ is a unit in $Q$ here; if it is, then such a $P$ is finitely generated and projective
as a module over $Q \# G$.) 
We write $[\mf_G(Q,f)]$ for the quotient of $\mf_G(Q,f)$ obtained by modding out by homotopy in the usual sense.

Our main result, Theorem~\ref{main}, is an immediate consequence of the following slightly stronger statement:

\begin{thm}
\label{prop:technical}
 Let $G$ be a finite group acting on a commutative ring $Q$, and assume $f \in Q$ is fixed by $G$. 
If $Q$ is (local noetherian) henselian, and the ring extension $Q^G \into Q$ is module finite, then $[\mf_G(Q,f)]$ is a Krull-Schmidt category and hence idempotent  complete. 
\end{thm}

\begin{proof} 
  Since $Q$ is henselian, $Q^G$ is also henselian by Proposition \ref{newprop}.
  Since $Q^G$ belongs to the center of $Q \# G$, the additive category $\mf_G(Q,f)$ is $Q^G$-linear.
    The endomorphism ring of every object $P$ of $\mf_G(Q,f)$ is contained in $\End_Q(P)$ and hence is module finite over $Q^G$. 
So,    since $[\mf_G(Q,f)]$ is a quotient of $\mf_G(Q,f)$, by Lemma \ref{lem126} it suffices to prove   $\mf_G(Q,f)$ is idempotent complete.

  Let $e$ be an idempotent endomorphism of an object $(P, d)$ in $\mf_G(Q,f)$.
 The category of all modules over $Q \# G$ is certainly idempotent complete, and so 
$P$ decomposes as  $P = \ker(e) \oplus \im(e)$ over this ring. 
  Since $P$ is $Q$-projective, so are both $\ker(e)$ and $\im(e)$.
  Since $e$ commutes with $d$, we have $d(\ker(e)) \subseteq \ker(e)$ and $d(\im(e)) \subseteq \im(e)$.
  Thus, $(\ker(e), d|_{\ker(e)})$ and $(\im(e), d|_{\im(e)})$ are objects of $\mf_G(Q,f)$,
and the canonical maps 
$p: (P, d) \onto (\im(e), d|_{\im(e)})$  and $i: (\im(e), d|_{\im(e)}) \into (P,d)$ are morphisms in $\mf_G(Q,f)$.
Since $e = i \circ p$,  this proves $\mf_G(Q,f)$  is idempotent complete.
\end{proof}

\section{Proofs of Proposition~\ref{introprop} and Corollary~\ref{introcor}}

Recall that, if $\cT$ is a triangulated category, a subcategory $\cS$ of $\cT$ is called {\em thick}
if $\cS$ is full, triangulated, and closed under summands. 
Given a collection $\cX$ of objects of $\cT$, the {\em thick closure} of $\cX$ in $\cT$, written $\Thick_{\cT}(\cX)$, is the intersection of all thick subcategories
of $\cT$ that contain $\cX$.  Let us say that an object $X$ of $\cT$ {\em builds} $\cT$ if $\Thick_{\cT}(\{X\}) = \cT$. Concretely, this means that every object of $\cT$ is obtained
from $X$ by a finite process of taking mapping cones, suspensions, and summands. 

Given a dg-category $\cC$, we write $[\cC]$ for its homotopy category, which has the same objects as $\cC$ and morphisms $\Hom_{[C]}(X, Y) \coloneqq H^0 \Hom_C(X,Y)$. 
We say $\cC$ is {\em pre-triangulated} if the image of the dg-Yoneda embedding $[\cC] \into [\Moddg(\cC)]$ is a triangulated subcategory of $[\Moddg(C)]$. See, e.g., \cite [Section 2.3]{Orlov}
for more details; roughly this means that $\cC$ has notions of suspension and mapping cone making  $[\cC]$ into a triangulated category. For example, the dg-category $\mf(Q,f)$ is pre-triangulated.

We use the following well-known fact:

\begin{lem} \label{lem1}
Suppose $\phi: \cC \to \cD$ is a dg-functor between two pre-triangulated dg-categories. Assume there exists an object $X \in \cC$  such that
\begin{itemize}
\item $X$ builds $[\cC]$,
  \item $\phi(X)$ builds $[\cD]$, and
\item the map $\phi: \End_{\cC}(X) \to \End_{\cC}(\phi(X))$ of dga's is a quasi-isomorphism.
\end{itemize}
The  dg-functor $\phi$ induces an equivalence $[\cC]^\vee \xra{\cong} [\cD]^\vee$ on idempotent completions of the associated homotopy categories. 
\end{lem}

\begin{proof}

  This essentially follows from \cite[Proposition 2.7]{Orlov}. In more detail: given a pre-triangulated dg-category $\cA$, let $\Perf(\cA)$ denote the (triangulated) homotopy category of the dg-category of perfect right $\cA$-modules; see e.g. \cite[Definition 2.3]{Orlov} and the surrounding discussion for additional background. As stated in e.g. \cite[\S 2.3]{Orlov}, the Yoneda embedding $[\cA] \into \Perf(\cA)$ exhibits $\Perf(\cA)$ as the idempotent completion of $[\cA]$. 
  
  We have the the following commutative diagram of triangulated categories:
  $$
  \xymatrix{
      \Perf(\End_\cC(X)) \ar[r]^-{\cong} \ar[d]^\cong & \Perf(\cC) \ar[d] &\ar[l]_-{\cong} [\cC]^\vee \ar[d] \\
      \Perf(\End_\cD(\phi(X))) \ar[r]^-{\cong} & \Perf(\cD) & \ar[l]_-{\cong} [\cD]^\vee;
    }
    $$
    the vertical maps are induced by $\phi$, the leftmost horizontal maps are induced by inclusions, and the rightmost maps are induced by the Yoneda embeddings. The rightmost horizontal maps are equivalences by the above discussion; since we assume $X$ builds $\cC$ and $\phi(X)$ builds $\cD$, \cite[Proposition 2.7]{Orlov} implies that the leftmost horizontal functors are equivalences as well.
 The leftmost vertical map is an equivalence since we assume $\End_\cC(X) \to \End_\cD(\phi(X))$ is a quasi-isomorphism.
Thus, the rightmost vertical map is an equivalence. 
  \end{proof}

Let $F: \mf_G(Q, f) \to \mf(Q,f)$ be the evident dg-functor that forgets the group action.
Since $Q \# G$ is free of finite rank as a $Q$-module, given $(P, d) \in \mf(Q,f)$, the pair $((Q \# G) \otimes_Q P, \id \otimes d)$ is an object of $\mf_G(Q,f)$. 
We extend this to a rule on morphisms in the evident way to obtain a dg-functor $E: \mf(Q, f) \to \mf_G(Q,f)$.

\begin{lem} \label{lem2} Assume $|G|$ is a unit in $Q$. Given $P \in \mf(Q,f)$, if $P$ builds $[\mf(Q,f)]$, then $E(P)$ builds $[\mf_G(Q,f)]$.
\end{lem}

\begin{proof} Given objects $X$ and $Y$ in a triangulated category $\cT$, we use the notation $X \builds_\cT Y$ as a shorthand for ``$X$ builds $Y$ (in $\cT$)".
The goal is to prove $E(P) \builds_{[\mf_G(Q,f)]} Y$ for all $Y \in [\mf_G(Q,f)]$. 
  For any such $Y$,  by assumption we have $P \builds_{[\mf(Q,f)]} F(Y)$. Since $E$ induces a triangulated functor on homotopy categories, it follows that
  $E(P) \builds_{[\mf_G(Q,f)]} E(F(Y))$. It therefore suffices to prove
  $E(F(Y)) \builds_{[\mf_G(Q,f)]} Y$; in fact, we show $Y$ is a summand of $E(F(Y))$ in $\mf_G(Q,f)$. 

The object $E(F(Y))$ has underlying module $(Q \# G) \otimes_Q Y$, with $G$-action through the left tensor factor (and the $G$ action on $Y$ ignored) and differential $\id \otimes d_Y$.
  There is an evident surjection $p: E(F(Y)) \onto Y$ in $\mf_G(Q,f)$ given by multiplication. Define
  $j: Y \into E(F(Y))$ by $j(y) = \frac{1}{|G|} \sum_{g \in G} g^{-1} \otimes y^g$. 
  One readily verifies that (a) $j$ is $Q \# G$ linear,
  (b) $j$ commutes with the differentials, and (c) $p \circ j = \id_Y$, so that $j$ is a splitting of $p$ in $\mf_G(Q,f)$. 
  \end{proof}

  \begin{proof}[Proof of Proposition \ref{introprop}]
Let $k = Q/\fm$ be the residue field of $Q$. For a sufficiently high $Q/f$-syzygy $M$ of $k$, we have that $M$ is a maximal Cohen-Macaulay (MCM) $Q/f$-module. By a Theorem of Eisenbud \cite{eisenbud}, the MCM module $M$ determines an object in $\mf(Q,f)$; let $k^\stab$ be such a matrix factorization.
(We note that the object $k^\stab$ depends on $M$ only up to a shift in $[\mf(Q, f)]$.)
Since  $Q/f$ has an isolated singularity, it follows from \cite[Corollary 4.12]{dyckerhoff} that $k^\stab$ builds $[\mf(Q,f)]$. 

Let us write $(k')^\stab$ for the image of $k^\stab$ in  $[\mf(Q',f')]$ under $\phi_*$; that is, $(k')^\stab = Q' \otimes_Q k^\stab$.
Since $\phi$ is flat, $(k')^\stab$ is the matrix factorization associated to the MCM $Q'/f'$-module $Q' \otimes_Q M$, which is
a high syzygy of $Q' \otimes_Q Q/\fm \cong Q'/\fm'$. Since $Q'/f'$ is an isolated singularity,
we see that $(k')^\stab$ builds $[\mf(Q',f')]$.

Set $X = E(k^\stab)$ and $X' = E((k')^\stab)$, where $E$ is the extension of scalars functor introduced above. 
By Lemma \ref{lem2}, we have that $X$ and $X'$  build $[\mf_G(Q, f)]$ and $ [\mf_G(Q', f')]$, respectively. 
Moreover, $X$ maps to $X'$ under the functor $[\mf_G(Q, f)] \to  [\mf_G(Q', f')]$. Since $\phi$ is flat, we have an isomorphism 
$$
Q' \otimes_Q  H_*(\End_{\mf_G(Q,f)}(X))  \cong H_*(\End_{\mf_G(Q',f')}(X'))
$$
of $Q'$-modules. 
Since the singularities are isolated, $H_*(\End_{\mf_G(Q,f)}(X))$ and $H_*(\End_{\mf_G(Q',f')}(X'))$ are finite length $Q$-modules,
and so, since the natural map $Q' \otimes_Q Q/\fm \to Q'/\fm'$
is an isomorphism, the natural map
$$
H_*(\End_{\mf_G(Q,f)}(X))  \xra{\cong} Q' \otimes_Q H_*(\End_{\mf_G(Q,f)}(X))
$$
is an isomorphism as well. It follows that the map of dga's
$$
\End_{\mf_G(Q,f)}(X)  \to \End_{\mf_G(Q',f')}(X') 
$$
is a quasi-isomorphism, so that Lemma \ref{lem1} yields an equivalence
$
[\mf_G(Q,f)]^\vee \xra{\cong} [\mf_G(Q',f')]^\vee.
$
\end{proof}

  \begin{proof}[Proof of Corollary~\ref{introcor}]    
By Proposition \ref{newprop}, each of $(Q^h)^G \subseteq Q^h$ and $\widehat{Q}^G \subseteq \widehat{Q}$ are module finite extensions. Proposition \ref{introprop} therefore gives equivalences
$$
[\mf_G(Q,f)]^\vee \xra{\cong} [\mf_G(Q^h,f)]^\vee \xra{\cong} [\mf_G(\widehat{Q},f)]^\vee;
$$
applying Theorem \ref{main} to both $[\mf_G(Q^h,f)]$ and $[\mf_G(\widehat{Q},f)]$ finishes the proof. 
\end{proof}

\section{Equivariant objects in the homotopy category of matrix factorizations}

Finally, we address a remark of Spellmann-Young in \cite{SY}. Let $[\mf(Q, f)]^G$ be the category of equivariant objects in $[\mf(Q, f)]$, as defined, for instance, by Carqueville-Runkel in \cite[\S 7.1]{CR}. There is a canonical functor
\begin{equation}
\label{canonical}
[\mf_G(Q, f)] \to [\mf(Q, f)]^G,
\end{equation}
and it is proven in \cite[Proposition 3.6]{SY} that, under certain circumstances, \eqref{canonical} exhibits the target as the idempotent completion of the source (in fact, this result applies more generally to Spellmann-Young's notion of Real equivariant matrix factorizations as well). Spellmann-Young note in \cite[Remark 3.7]{SY} that, if the map \eqref{canonical} (or its Real generalization) were an equivalence, some of their arguments could be shortened; we now apply Theorem~\ref{main} to prove this.

\begin{cor}
\label{fixedpoints}
Suppose we are in the setting of Theorem~\ref{main}, and assume further that $|G|$ is a unit in $Q$. The functor \eqref{canonical} is an equivalence. 
\end{cor}

\begin{proof}
The proof of \cite[Proposition 3.6]{SY} extends to our setting. In detail: by Theorem~\ref{main}, the triangulated category $[\mf_G(Q, f)]$ is idempotent complete. As discussed in the proof of Lemma~\ref{lem1}, the Yoneda embedding
$[\mf_G(Q, f)] \into \Perf(\mf_G(Q, f))$ is an idempotent completion and hence an equivalence in our case. Since $|G|$ is a unit in $Q$, \cite[Theorem
8.7]{elagin2} implies that there is an equivalence $\Perf(\mf_G(Q, f)) \to [\mf(Q, f)]^G$. Composing these two equivalences gives \eqref{canonical}.
\end{proof}

\begin{rem}
While it is assumed in \cite[Proposition 3.6]{SY} that the local hypersurface ring under consideration has an isolated singularity, this assumption is not necessary for Corollary~\ref{fixedpoints}. 
\end{rem}

\bibliographystyle{amsplain}
\bibliography{Bibliography}
\end{document}